\documentclass[times]{amsart}

\usepackage{latexsym}
\usepackage{graphics}
\usepackage{amssymb}
\usepackage{amsmath}
\usepackage{amscd}
\usepackage{amsthm}
\usepackage{tikz}

\newtheorem{theorem}{Theorem}[section]

\newtheorem{prop}[theorem]{Proposition}

\newtheorem{ex}[theorem]{Example}
\newtheorem{defn}[theorem]{Definition}

\begin{document}

\title{Harmonic Galois theory for finite graphs}
\author{Scott Corry}

\address{Lawrence University, Department of Mathematics, 711 E. Boldt Way -- SPC 24, Appleton, WI 54911, USA}
\email{corrys@lawrence.edu}

\maketitle

\begin{abstract} This paper develops a harmonic Galois theory for finite graphs, thereby classifying harmonic branched $G$-covers of a fixed base $X$ in terms of homomorphisms from a suitable fundamental group of $X$ together with $G$-inertia structures on $X$. As applications, we show that finite embedding problems for graphs have proper solutions and prove a Grunwald-Wang type result stating that an arbitrary collection of fibers may be realized by a global cover.
\end{abstract}

\section{Introduction}

The fact that finite graphs may be viewed as discrete Riemann surfaces has appeared in a variety of contexts, and this analogy has connections to arithmetic geometry, tropical geometry, and cryptography. Baker and Norine \cite{BNHarm} introduced harmonic morphisms as the correct graph-analogue of holomorphic maps and derived a harmonic Riemann-Hurwitz formula. They also studied hyperelliptic graphs: finite graphs of genus at least 2 possessing an involution such that the quotient morphism is harmonic with target a tree. In \cite{C} we furthered this line of thought by introducing the notion of a harmonic group action on a finite graph, and determined sharp linear genus bounds on the maximal size of such actions. The main result of \cite{C} is a graph-analogue of the Accola-Maclachlan \cite{A}, \cite{Mac} and Hurwitz \cite{H} genus bounds for holomorphic group actions on compact Riemann surfaces.

Our approach in \cite{C} was ``top-down'' in the sense that we began with a finite graph $Y$ and searched for groups $G$ acting harmonically on $Y$, our main tool being the Riemann-Hurwitz formula applied to the harmonic quotient morphism $Y\rightarrow G \backslash Y$. But Galois theory (whether for Riemann surfaces, algebraic curves, schemes, topological spaces, etc.) generally proceeds in a ``bottom-up'' fashion, fixing a base object $X$, and studying a distinguished class of ``covers'' of $X$ with the goal of classifying these covers by means of a suitable fundamental group (see~\cite{Sz} for a broad overview). In the case of Riemann surfaces (and more generally algebraic curves), this approach succeeds in classifying branched covers in addition to \'etale covers. The aim of this paper is to develop such a harmonic Galois theory for finite graphs, and our basic goal is as follows: given a finite graph $X$, a subset $B\subset V(X)$, and a finite group $G$, to describe as precisely as possible the harmonic $G$-covers of $X$ branched only at $B$. In section~\ref{EP}, we present such a description in the form of a Grunwald-Wang type theorem (Theorem~\ref{GW}).

As observed in section 4.1 of \cite{C}, arbitrary finite groups occur as inertia groups of harmonic branched covers of graphs, unlike the case of Riemann surfaces where all inertia is cyclic. Moreover, while horizontal ramification for graphs corresponds nicely to the ramification of Riemann surfaces, there is no analogue of vertical ramification in the classical context. In light of these differences, we wish to promote the idea that graphs are actually analogous to smooth proper algebraic curves over a perfect non-algebraically closed field. If $C\rightarrow D$ is a degree-$n$ \'etale cover of such curves, then a scheme-theoretic fiber $C_x$ may have fewer than $n$ points provided there is an extension of residue fields. If the cover is Galois, this corresponds to the presence of non-trivial decomposition groups with trivial inertia groups. Analogously, if $Y\rightarrow X$ is a harmonic $G$-cover of finite graphs with only vertical ramification, then each fiber $Y_x$ has $|G|$ vertices, but the vertical ramification produces fewer than $|G|$ connected components, and the stabilizers of these components play the role of decomposition groups. Motivated by this analogy, we will refer to horizontally unramified harmonic maps as graph-theoretic \'etale covers (see Definition~\ref{DI}).

Another difference between graphs and Riemann surfaces is that there is no full harmonic automorphism group of a finite graph. Moreover, if $Y\rightarrow X$ is a harmonic morphism, then there may exist two nonisomorphic subgroups $G_1,G_2\le \textrm{Aut}(Y|X)$ such that the natural maps $G_i \backslash Y\rightarrow X$ are isomorphisms (see Examples~\ref{S3}, \ref{Z6}). This means that we must specify the $G$-action on $Y$ as an $X$-graph as part of the data of our branched covers. We introduce the relevant categories of $G$-covers in Definitions~\ref{HarmDef}--\ref{EtDef} after reviewing the notion of harmonic group action from \cite{C}.

This paper began as a lecture delivered during the RIMS-Camp-Style Seminar ``Galois-theoretic Arithmetic Geometry'' held in Kyoto, Japan (October 19-24, 2010). The author thanks the organizers as well as the referee who provided many valuable comments. In addition, the author was supported by National Science Foundation grant DMS-1044746 for travel and conference expenses.

\subsection{Terminology and notation}\label{TN}

The term \emph{graph} will always refer to a finite multi-graph without loop edges. This means that two vertices of a graph may be connected by multiple edges, but no vertex has an edge to itself. While these are the main objects of interest, multi-graphs with (possibly infinitely many) loops will play an auxiliary role in our constructions. We will refer to these more general graphs as \emph{loop-graphs} to avoid confusion. For a (loop)-graph $X$, we denote by $V(X)$ and $E(X)$ the vertex- and edge-sets of $X$ respectively. For $x\in V(X)$, the subgraph of $X$ induced by the edges incident to $x$ is denoted $x(1)$, and should be thought of as the smallest neighborhood of $x$ in $X$. To be explicit, the vertices of $x(1)$ are $x$ together with the vertices of $X$ adjacent to $x$, and the edges of $x(1)$ are those of $X$ incident to $x$. The \emph{genus} of a connected graph $X$ is the rank of its first Betti homology group: $g(X)=|E(X)|-|V(X)|+1$. As in \cite{BNHarm}, \cite{C}, this language is chosen to emphasize the analogy with Riemann surfaces, despite the fact that the quantity $g(X)$ is more commonly called the cyclomatic number of $X$ by graph theorists, who use the term genus for a different concept (i.e. the minimal topological genus of a surface into which $X$ may be embedded).

\section{Harmonic group actions}\label{HG}

We begin by recalling the definition of a harmonic morphism between graphs and some related terminology from \cite{BNHarm}. Note that in the following definition, we allow the graphs to be disconnected.
\begin{defn} A \emph{morphism of graphs} $\phi:Y\rightarrow X$ is a function $\phi:V(Y)\cup E(Y)\rightarrow V(X)\cup E(X)$ mapping vertices to vertices and such that for each edge $e\in E(Y)$ with endpoints $y_1\ne y_2$, either $\phi(e)\in E(X)$ has endpoints $\phi(y_1)\ne \phi(y_2)$, or $\phi(e)=\phi(y_1)=\phi(y_2)\in V(X)$. In the latter case, we say that the edge $e$ is \emph{$\phi$-vertical}. $\phi$ is \emph{degenerate} at $y\in V(Y)$ if $\phi(y(1))=\{\phi(y)\}$, i.e. if $\phi$ collapses a neighborhood of $y$ to a vertex of $X$. The morphism $\phi$ is \emph{harmonic} if for all vertices $y\in V(Y)$, the quantity $|\phi^{-1}(e')\cap y(1)|$ is independent of the choice of edge $e'\in E(\phi(y)(1))$. 
\end{defn}

In section~\ref{etale}, we will need to consider morphisms between loop-graphs (see section~\ref{TN}), which for completeness we now define explicitly. The key point is that whereas graph morphisms must contract an edge whose endpoints are mapped to the same vertex, loop-graph morphisms can send such an edge to a loop.

\begin{defn}
A \emph{morphism of loop-graphs} $\psi:W\rightarrow Z$ is a function $\psi:V(W)\cup E(W)\rightarrow V(Z)\cup E(Z)$ mapping vertices to vertices and such that for each edge $e\in E(W)$ with endpoints $w_1,w_2$, either $\psi(e)\in E(Z)$ has endpoints $\psi(w_1),\psi(w_2)$, or $\psi(e)=\psi(w_1)=\psi(w_2)\in V(Z)$. 
\end{defn}

\begin{defn}\label{deg}
Let $\phi:Y\rightarrow X$ be a harmonic morphism between connected graphs. If $|V(X)|>1$ (i.e. if $X$ is not the point graph $\star$), then
the \emph{degree} of the harmonic morphism $\phi$ is the number of pre-images in $Y$ of any edge of $X$ (this is well-defined by \cite{BNHarm}, Lemma 2.4). If $X=\star$ is the point graph, then the \emph{degree} of $\phi$ is defined to be $|V(Y)|$, the number of vertices of $Y$.
\end{defn}

In \cite{BNHarm}, the authors define the degree of any harmonic morphism to the point graph $\star$ to be zero. Since such morphisms play a central role in our theory, we need to alter this convention as in Definition~\ref{deg}, especially in the context of vertex-transitive group actions on graphs (see Example~\ref{Cay}). Note, however, that according to our definition, a constant harmonic morphism to a connected graph with more than one vertex still has degree zero, as in \cite{BNHarm}.

\begin{defn}
Suppose that $G\le\textrm{Aut}(Y)$ is a (necessarily finite) group of automorphisms of the graph $Y$, so that we have a left action $G\times Y\rightarrow Y$ of $G$ on $Y$. We say that $(G,Y)$ is a \emph{faithful group action} if the stabilizer of each connected component of $Y$ acts faithfully on that component. Note that this condition is automatic if $Y$ is connected.
\end{defn} 

Given a faithful group action $(G,Y)$, we denote by $G\backslash Y$ the quotient graph\footnote{The notation $Y/G$ is used for the quotient graph in \cite{BNHarm} and \cite{C}. We have chosen the notation $G\backslash Y$ as in \cite{Sz} to emphasize the left action of $G$ on $Y$.} with vertex-set $V(G\backslash Y)=G\backslash V(Y)$, and edge-set 
$$
E(G\backslash Y)=G\backslash E(Y) - \{Ge \ | \ \textrm{$e$ has endpoints $y_1,y_2$ and $Gy_1=Gy_2$}\}.
$$
Thus, the vertices and edges of $G\backslash Y$ are the left $G$-orbits of the vertices and edges of $Y$, with any loop edges removed. There is a natural morphism $\phi_G:Y\rightarrow G\backslash Y$ sending each vertex and edge to its $G$-orbit, and such that edges of $Y$ with endpoints in the same $G$-orbit are $\phi_G$-vertical.

The observation that the quotient morphism $\phi_G$ need not be harmonic in general led us to introduce the notion of a harmonic group action in \cite{C}.
\begin{defn}\label{HA} Suppose that $(G,Y)$ is a faithful group action. Then $(G,Y)$ is a \emph{harmonic group action} if for all subgroups $H<G$, the quotient morphism $\phi_H:Y\rightarrow H\backslash Y$ is harmonic.
\end{defn}
The original definition (\cite{C}, Definition 2.4) included a non-degeneracy requirement for harmonic morphisms; we explictly relax that requirement in the present paper. Nevertheless, the following proposition shows that only a very specific type of degeneracy is possible for harmonic group actions with connected quotients. See Example~\ref{Cay} for an illustration.

\begin{prop}
Suppose that $(G,Y)$ is a harmonic group action such that $G\backslash Y$ is connected. If the quotient morphism $\phi_G:Y\rightarrow G\backslash Y$ is degenerate, then $G\backslash Y$ is the point graph $\star$. 
\end{prop}

\begin{proof}
Suppose that $\phi_G$ is degenerate at $y\in V(Y)$, and set $x=\phi_G(y)$. Then $\phi_G$ is degenerate at every vertex of the fiber $Y_x$. But then $\phi_G^{-1}(e)=\emptyset$ for every edge $e\in E(x(1))$. Since $\phi_G$ is surjective, it follows that $E(x(1))=\emptyset$. Connectivity then implies $G\backslash Y=\star$.
\end{proof}

Proposition 2.5 of \cite{C} provides a criterion for a group action on a connected graph to be harmonic and non-degenerate, and a small modification of the proof shows that the following version holds for possibly degenerate harmonic group actions on possibly disconnected graphs.
\begin{prop}\label{Crit}
Suppose that $(G,Y)$ is a faithful group action. Then $(G,Y)$ is a harmonic group action if and only if for every vertex $y\in V(Y)$, the stabilizer subgroup $G_y$ acts freely on the edge-set $E(y(1))$. 
\end{prop}

An equivalent way of stating the previous criterion is that $(G,Y)$ is harmonic if and only if the stabilizers of \emph{directed} edges are trivial. Note, however, that a \emph{non-directed} edge $e$ can be fixed by some element $\tau\in G$, provided that $\tau$ switches the endpoints of $e$. But then $\tau^2$ fixes the directed edge $e$, forcing it to be the identity. Similarly, if $\tau'$ is another involution flipping $e$, then $\tau\tau'$ fixes the directed edge $e$, forcing $\tau=\tau'$. This shows that non-trivial edge-stablizers of harmonic group actions have order 2. If all edge-stabilizers are trivial, then we say that the harmonic group action is \emph{unflipped}. 

We now describe a simple procedure that produces an unflipped harmonic group action from an arbitrary harmonic group action $(G,Y)$. If $e$ is a flipped edge in $(G,Y)$, then by the orbit-stabilizer theorem, the orbit $Ge$ has $\frac{|G|}{2}$ elements. Doubling each of the edges in $Ge$ yields a new graph $Y'$ containing $Y$ as a subgraph. Moreover, $G$ acts harmonically on $Y'$ in such a way that $e$ is unflipped. Repeating this construction for any remaining flipped edges ultimately produces an unflipped harmonic group action $(G,\tilde{Y})$. Moreover, the graph $\tilde{Y}$ is minimal with respect to the property that $\tilde{Y}$ is obtained from $Y$ by adding edges. In this way we see that every harmonic group action $(G,Y)$ correponds to a unique unflipped action $(G,\tilde{Y})$, called the $\emph{unflipped model}$ of $(G,Y)$. Conversely, given an unflipped action $(G,\tilde{Y})$, it is straightforward to recover all pairs $(G,Y)$ whose unflipped model is $(G,\tilde{Y})$ by looking at the way conjugacy classes of involutions in $G$ act on $\tilde{Y}$. Thus, there is no loss in restricting attention to unflipped harmonic group actions; the fact that such actions are fixed-point free away from the vertices (when viewed as maps of topological spaces) will be essential to our analysis in section~\ref{etale}.

The following central example explains our desire to weaken the definition of harmonic group action as in Definition~\ref{HA} by removing the requirement of non-degeneracy.
\begin{ex}[Cayley graphs]\label{Cay} Let $G$ be a finite group, and $S=\{\delta_i\}$ a finite symmetric multi-set of elements of $G$. Thus, the elements of $S$ may occur with multiplicity, and $S$ is stable under inversion of group elements.  We define a (possibly disconnected) \emph{Cayley graph} $\textrm{Cay}(G,S)$ with vertex set $G$ as follows: for each $\delta_i\in S$, and each vertex $g\in G$, there is an edge connecting $g$ to $g\delta_i$. Moreover, if $\delta_i\ne \delta_i^{-1}$, we identify this edge to the edge connecting $g\delta_i$ to $g=g\delta_i\delta_i^{-1}$ corresponding to $\delta_i^{-1}\in S$. We \emph{do not} identify edges corresponding to involutions in $S$. The graph $\textrm{Cay}(G,S)$ will be connected precisely when $G$ is generated by the subset $S$. In any case, $\textrm{Cay}(G,S)$ supports a natural unflipped $G$-action given by left multiplication in $G$. Since the vertex stabilizers are trivial, Proposition~\ref{Crit} shows that the action is harmonic. It is also vertex-transitive, so the quotient graph is a single point,  and the quotient morphism $\phi_G:\textrm{Cay}(G,S)\rightarrow\star$ is degenerate. The degree of $\phi_G$ is $|G|$ according to Definition~\ref{deg}.
\end{ex}

\begin{ex}[A Cayley graph on $\mathfrak{S}_3$]\label{S3}
Consider the symmetric group $\mathfrak{S}_3=\left<\sigma,\tau \ | \ \sigma^3=\tau^2=\varepsilon, \sigma\tau=\tau\sigma^2\right>$, together with the generating set $S=\{\sigma,\sigma^{-1},\tau\}$. The Cayley graph $\textrm{Cay}(\mathfrak{S}_3,S)$ is shown below:

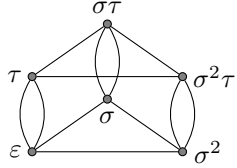
\begin{figure}[h]
\centering
\begin{tikzpicture}
\tikzstyle{every node}=[circle, draw, fill=black!50, inner sep=0pt, minimum width=3pt]

\node (1) at  (-1,1)  {};
\coordinate [label=left : $\tau$] (tau) at (-1.1,1);
\node (2) at  (0,1.7) {} ;  
\coordinate [label=above : $\sigma\tau$] (sigmatau) at (0,1.7);      
\node (3) at  (1,1) {};
\coordinate [label=right : $\sigma^2\tau$] (sigma2tau) at (1.1,1);

\node (4) at  (-1,0)  {};
\coordinate [label=left : $\varepsilon$] (epsilon) at (-1.1,0);
\node (5) at  (0,.7) {} ;  
\coordinate [label=below : $\sigma$] (sigma) at (0,.6);      
\node (6) at  (1,0) {};
\coordinate [label=right : $\sigma^2$] (sigma2) at (1.1,0);

\draw[-] (1) to (2);
\draw[-] (2) to (3);
\draw [-] (3) to (1);

\draw[-] (4) to (5);
\draw[-] (5) to (6);
\draw[-] (6) to (4);

\draw[-] (1) to [out=240, in=120] (4);
\draw[-] (1) to [out=300, in=60] (4);
\draw[-] (2) to [out=240, in=120] (5);
\draw[-] (2) to [out=300, in=60] (5);
\draw[-] (3) to [out=240, in=120] (6);
\draw[-] (3) to [out=300, in=60] (6);

\end{tikzpicture}
\end{figure}
\noindent
The single horizontal edges correspond to the pair $\{\sigma,\sigma^{-1}\}\subset S$, while the double vertical edges correspond to the involution $\tau\in S$. The $\mathfrak{S}_3$-action is given by left multiplication on the labeled vertices.
\end{ex}

\begin{ex}[A Cayley graph on $\mathbb{Z}/6\mathbb{Z}$]\label{Z6}
Consider the cyclic group $\mathbb{Z}/6\mathbb{Z}=\left<\alpha \ | \ \alpha^6=\varepsilon\right>$, together with the generating set $S=\{\alpha^2,\alpha^{-2},\alpha^3\}$. The Cayley graph $\textrm{Cay}(\mathbb{Z}/6\mathbb{Z},S)$ is shown below:

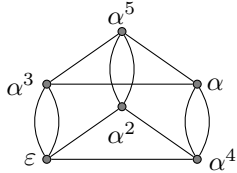
\begin{figure}[h]
\centering
\begin{tikzpicture}
\tikzstyle{every node}=[circle, draw, fill=black!50, inner sep=0pt, minimum width=3pt]

\node (1) at  (-1,1)  {};
\coordinate [label=left : $\alpha^3$] (a3) at (-1.1,1);
\node (2) at  (0,1.7) {} ;  
\coordinate [label=above : $\alpha^5$] (a5) at (0,1.7);      
\node (3) at  (1,1) {};
\coordinate [label=right : $\alpha$] (a) at (1.1,1);

\node (4) at  (-1,0)  {};
\coordinate [label=left : $\varepsilon$] (epsilon) at (-1.1,0);
\node (5) at  (0,.7) {} ;  
\coordinate [label=below : $\alpha^2$] (a2) at (0,.6);      
\node (6) at  (1,0) {};
\coordinate [label=right : $\alpha^4$] (a4) at (1.1,0);

\draw[-] (1) to (2);
\draw[-] (2) to (3);
\draw [-] (3) to (1);

\draw[-] (4) to (5);
\draw[-] (5) to (6);
\draw[-] (6) to (4);

\draw[-] (1) to [out=240, in=120] (4);
\draw[-] (1) to [out=300, in=60] (4);
\draw[-] (2) to [out=240, in=120] (5);
\draw[-] (2) to [out=300, in=60] (5);
\draw[-] (3) to [out=240, in=120] (6);
\draw[-] (3) to [out=300, in=60] (6);

\end{tikzpicture}
\end{figure}
\noindent
The single horizontal edges correspond to the pair $\{\alpha^2,\alpha^{-2}\}\subset S$, while the double vertical edges correspond to the involution $\alpha^3\in S$. The $\mathbb{Z}/6\mathbb{Z}$-action is given by left multiplication on the labeled vertices.

\end{ex}

These examples show that the same graph $Y$ may be a Cayley graph for two different groups. In particular, the group $G$ cannot be recovered from the quotient morphism $\phi_G:Y\rightarrow \star$. This phenomenon has nothing to do with the target being a point: a harmonic group action $(G,Y)$ is not in general determined by the quotient morphism $Y\rightarrow G\backslash Y$. For an example where the target is a segment (rather than a point), consider the ``cone'' on the graph in Examples~\ref{S3} and \ref{Z6}:

\begin{figure}[h]
\centering
\begin{tikzpicture}
\tikzstyle{every node}=[circle, draw, fill=black!50, inner sep=0pt, minimum width=3pt]

\node (1) at  (-1,1)  {};
\node (2) at  (0,1.7) {} ;  
\node (3) at  (1,1) {};

\node (4) at  (-1,0)  {};
\node (5) at  (0,.7) {} ;  
\node (6) at  (1,0) {};

\node (7) at (-3,.85) {};

\draw[-] (1) to (2);
\draw[-] (2) to (3);
\draw [-] (3) to (1);

\draw[-] (4) to (5);
\draw[-] (5) to (6);
\draw[-] (6) to (4);

\draw[-] (1) to [out=240, in=120] (4);
\draw[-] (1) to [out=300, in=60] (4);
\draw[-] (2) to [out=240, in=120] (5);
\draw[-] (2) to [out=300, in=60] (5);
\draw[-] (3) to [out=240, in=120] (6);
\draw[-] (3) to [out=300, in=60] (6);

\draw[-] (7) to (1);
\draw[-] (7) to (2);
\draw[-] (7) to [out=45, in=90] (3);
\draw[-] (7) to (4);
\draw[-] (7) to (5);
\draw[-] (7) to [out=-45, in=-135] (6);

\end{tikzpicture}
\end{figure}
\noindent
The $\mathfrak{S}_3$ and $\mathbb{Z}/6\mathbb{Z}$-actions from the examples extend to this graph, and both actions yield the same quotient morphism to a segment, so the group cannot be recovered from the quotient morphism alone. Hence, we study harmonic $G$-covers of a connected graph $X$ as defined below, rather than ``Galois covers'' of $X$. Note that in the following definition and throughout the rest of the paper the $G$-covers under consideration are required to be connected; this corresponds to the fact that Galois covers in topology are connected by definition.

\begin{defn}\label{HarmDef} Let $\textbf{Harm}$ be the category with objects
$$
\{(G,Y,y) \ | \ (G,Y) \textrm{ unflipped harmonic action, Y connected, $y\in V(Y)$}\},
$$
and morphisms
\begin{multline*}
\textrm{Hom}_{\textbf{Harm}}((G,Y,y), (G',Y',y')):=\\
\{(\varphi:G\rightarrow G', f:(Y,y)\rightarrow (Y',y')) \ | \
\varphi(g)f(\alpha)=f(g\alpha) \ \forall g\in G, \alpha\in Y\}.
\end{multline*}
Here $\varphi$ is a homomorphism of groups, $f$ is a harmonic morphism of pointed graphs, and $\alpha$ runs over both the vertices and edges of $Y$.
\end{defn}

The category \textbf{CG} of pointed connected graphs with harmonic morphisms embeds fully into the category \textbf{Harm} by sending $(X,x)$ to the trivial action $(\{\textrm{id}_X\}, X,x)$. Using this identification, we define the category $\textbf{Harm}_{(X,x)}$ of \emph{harmonic $G$-covers of $X$} as follows.

\begin{defn}\label{HarmXDef} Let $(X,x)$ be a pointed connected graph. Then define $\textbf{Harm}_{(X,x)}$
to be the full subcategory of the slice category of \textbf{Harm} over $(X,x)$ with objects
$$
\{f:(G,Y,y)\rightarrow (X,x) \ |  \ \overline{f}:G\backslash Y\rightarrow X \textrm{ is an isomorphism}\}.
$$
\end{defn}

\begin{defn}\label{DI}
Suppose that $f:(G,Y,y)\rightarrow (X,x)$ is a harmonic $G$-cover of $X$, and $w\in V(Y)$. The \emph{decomposition group} $\Delta_w$ at $w$ is the stabilizer of the connected component of the fiber $Y_{f(w)}$ containing $w$. The \emph{inertia group} $I_w$ at $w$ is the stabilizer subgroup of $w$ in $G$. Note that $I_w\le \Delta_w$, and the decomposition / inertia groups form conjugacy classes in $G$ as $w$ varies over the fiber $Y_{f(w)}$. We say that $f$ is \emph{horizontally unramified} or \emph{\'etale at $w$} if $I_w=\{\varepsilon\}$, and the cover $f$ is \emph{\'etale} if it is \'etale at all $w\in V(Y)$. If $\Delta_{w}=\{\varepsilon\}$ (resp. $I_w=G$) we say that $f(w)$ is \emph{totally split} (resp. \emph{totally ramified}) in $Y$.
\end{defn}

As mentioned in the Introduction, our definition of decomposition and inertia groups is motivated by the theory of algebraic curves over perfect non-algebraically closed fields. If $C\rightarrow D$ is a degree-$n$ Galois cover of such curves, then the fiber $C_x$ over a point $x\in D$ may have fewer than $n$ points for two different reasons: ramification and extension of residue fields. Moreover, the relationship between these two possibilities is captured by the decomposition and inertia groups. For $y\in C_x$, the decomposition group at $y$ is defined to be the stabilizer subgroup of $y$ in $\textrm{Gal}(C|D)$, while the inertia subgroup is defined to be the subgroup of the decomposition group that acts trivially on the residue field at $y$. Ramification corresponds to a nontrivial inertia group, while an extension of residue fields corresponds to the inertia being a \emph{proper} subgroup of the decomposition group. Of course, over an algebraically closed field like $\mathbb{C}$ (corresponding to the case of Riemann surfaces), no residue field extension is possible, and the decomposition and inertia groups are identical. By analogy, our graph-theoretic definitions are motivated by the idea that the connected components of the fiber $Y_{f(w)}$ are the ``points'' of the fiber, which allows for a new perspective on the phenomena of ``vertical ramification'': it is not ramification at all, but rather the graph-theoretic analogue of an extension of residue fields. Since \'etale (i.e. unramified) covers of curves are classified by the algebraic fundamental group, our analogy suggests that \'etale $G$-covers of graphs should be classified by a suitable fundamental group (see section~\ref{Et}).

\begin{defn}\label{EtDef} For a pointed connected graph $(X,x)$, we denote by $\textbf{\'Et}_{(X,x)}$ the full subcategory of $\textbf{Harm}_{(X,x)}$ consisting of \'etale $G$-covers of $X$.
\end{defn}

\subsection{Spanning trees}\label{trees}

A Riemann surface of genus $g$ is given by specifying a complex structure on the unique orientable topological surface of genus $g$. In the case of graphs, the best we can say is that every connected graph $X$ of genus $g$ is homotopy equivalent to the rose $R_g$ consisting of one vertex with $g$ loop edges. Such a homotopy equivalence induces an isomorphism $\pi_1(X)\cong\pi_1(R_g)\cong F_g$, the free group on $g$ generators, thereby yielding a concrete description of unramified covers of $X$. But a homotopy equivalence $X\rightarrow R_g$ will generally collapse many vertices of $X$ to the unique vertex of $R_g$, thereby destroying the notion of distinguished branch points. Moreover, whereas in the case of Riemann surfaces we can capture the branching by puncturing the surface, the analogous strategy for graphs fails. Indeed, puncturing a surface adds a free generator to the fundamental group (explaining why all inertia is cyclic), while removing a vertex from a graph will (if anything) \emph{decrease} the size of the fundamental group by killing off cycles. As we will see, rather than puncturing a surface to allow for ramification above a branch point, we will add a countably infinite wedge of loops at a vertex of $X$ to allow for ``vertical ramification'' (which we have argued above should be thought of as analogous to an extension of residue fields rather than a type of ramification). But this realization, while important, doesn't change the fact that any fundamental group that hopes to classify harmonic covers will need to see the difference between vertices of $X$.

Our solution to this problem is to choose a spanning tree $T\subset X$, so that $X$ may be thought of as the tree $T$ together with the extra structure of a multi-set of $g$ pairs of vertices specifying the edges of $X-T$. Of course, the spanning tree $T$ is not uniquely determined by $X$. Indeed, it is exactly this lack of uniqueness that forces us to fix a spanning tree $T\subset X$ in the subsequent development. The importance of fixing the spanning tree $T$ for the pointed graph $(X,x)$ comes from the following observation: if $f:(Y,y)\rightarrow (X,x)$ is a non-degenerate harmonic morphism, then the tree $T$ lifts (non-uniquely) to a tree $\tilde{T}$ in $Y$ containing $y$.  If $f$ is horizontally unramified, then the lifting is unique. In any case, the tree $\tilde{T}$ determines a vertex section $V(X)\rightarrow V(Y)$ to the map $f$, which we will denote by $z\mapsto\tilde{z}$. 

\section{The \'etale fundamental group}\label{etale}

Fix a connected pointed graph $(X,x)$. In this section we describe the category $\textbf{\'Et}_{(X,x)}$ by means of a suitable fundamental group, which we will call \emph{the \'etale fundamental group} of the pointed graph $(X,x)$ and denote by $\pi_1^{\textrm{\'et}}(X,x)$. The upshot, as described more fully in the next section, is that the structure of this fundamental group provides a concrete description of the \'etale $G$-covers of $X$. Namely, we will see that to give a pointed \'etale $G$-cover $(G,Y,y)\rightarrow (X,x)$ is to give a homomorphism $\pi_1(X,x)\rightarrow G$ (i.e. a pointed, topological $G$-cover of $X$) together with a finite, symmetric, unordered multi-set of nontrivial elements of $G$ at each vertex of $X$. Before embarking on the construction of the \'etale fundamental group, we briefly describe the steps of the argument as an aid to the reader:

\begin{enumerate}
\item Associate to each \'etale $G$-cover of $X$ a Galois topological cover of a loop-graph $X_n$.
\item Show that this association yields an isomorphism between a certain subcategory of \'etale $G$-covers $\textbf{\'Et}_{(X,x)}(n)$ and a category of Galois topological covers of $X_n$.
\item Take the direct limit to obtain an isomorphism between $\textbf{\'Et}_{(X,x)}$ and a category of Galois topological covers of a loop-graph $X_\infty$.
\item Use the theory of fiber functors to obtain a profinite group classifying the appropriate Galois topological covers of $X_\infty$. For an overview of the use of fiber functors in Galois theory, see \cite{Sz}.
\end{enumerate}
As described in section~\ref{trees}, we fix a spanning tree $T\subset X$ once and for all.

\noindent
\underline{Step 1:} Given an object $f:(G,Y,y)\rightarrow (X,x)$ of $\textbf{\'Et}_{(X,x)}$, let $\widetilde{G\backslash Y}$ denote the \emph{uncontracted} quotient, which may have loops. We have the following commutative diagram, where $X_f$ is a loop-graph obtained from $X$ by adding finitely many loops at each vertex, 
$\tilde{f}$ is an isomorphism extending $\overline{f}$, and the bottom vertical arrows are loop contractions:
$$
\begin{CD}
Y	   						@=						Y\\
@VVV              											@VVV\\
\widetilde{G\backslash Y}                   @>\tilde{f}>>				X_f\\
@VVV												@VVV\\
G\backslash Y					@>\overline{f}>>			X.			
\end{CD}
$$
Note that the isomorphism $\tilde{f}$ is unique up to a permutation of the loops of $X_f$, i.e. up to the action of the ``loop group'' $L_f:=\prod_{z\in V(X)}\mathfrak{S}_{n_z}$ on $X_f$, where
 \begin{eqnarray*}
n_z &:=& \textrm{ number of loops at the vertex } z\in X_f\\
&=& \textrm{ number of loops at the vertex } \overline{f}^{-1}(z)\in V(\widetilde{G\backslash Y}).
 \end{eqnarray*}
We thus obtain a map $Y\rightarrow X_f$, uniquely defined up to the action of $L_f$. Note that this map is purely combinatorial, but may be viewed as a topological $G$-cover of 1-dimensional CW complexes once we choose an orientation for each loop of $X_f$. Hence, the topological $G$-cover $Y\rightarrow X_f$ is defined up to the action of the ``topological loop group'' $TL_f:=\prod_{z\in V(X)}\left(\mathbb{Z}/2\mathbb{Z}\right)^{n_z}\rtimes L_f$, where the elementary abelian 2-groups account for the choice of orientations.

Now choose $n\ge\max\{n_z \ | \ z\in V(X)\}$, and let $X_n$ be the loop-graph obtained from $X$ by adding $n$ loops at each vertex. We have a surjective contraction map $X_n\rightarrow X_f$, which is well-defined up to the action of the topological loop group $TL_n$ on $X_n$. Pulling back $Y\rightarrow X_f$ over this contraction map yields a topological $G$-cover $Y_f\rightarrow X_n$, which is well-defined up to the action of $TL_n$ on $X_n$. This fits into the previous diagram as follows:
$$
\begin{CD}
Y	   				  @=				Y		@<<<	Y_f\\
@VVV              						   @VVV				@VVV\\
\widetilde{G\backslash Y}   @>\tilde{f}>>		X_f		@<<<	X_n\\
@VVV							   @VVV\\
G\backslash Y		            @>\overline{f}>>	X.			
\end{CD}
$$

To illustrate, consider the following $\mathbb{Z}/2\mathbb{Z}$-cover:

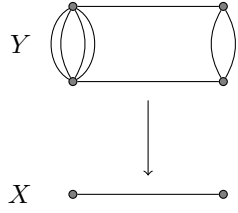
\begin{figure}[h]
\centering
\begin{tikzpicture}
\tikzstyle{every node}=[circle, draw, fill=black!50, inner sep=0pt, minimum width=3pt]

\node (1) at  (-1,1)  {};
\coordinate [label=left : $Y$] (Y) at (-1.5,.5);
\node (3) at  (1,1) {};

\node (4) at  (-1,0)  {};
\node (6) at  (1,0) {};

\draw [-] (3) to (1);

\draw[-] (6) to (4);

\draw[-] (1) to [out=240, in=120] (4);
\draw[-] (1) to [out=300, in=60] (4);
\draw[-] (1) to [out=210, in=150] (4);
\draw[-] (1) to [out=330, in=30] (4);
\draw[-] (3) to [out=240, in=120] (6);
\draw[-] (3) to [out=300, in=60] (6);

\draw[->] (0,-.25) -- (0,-1.25);

\node (2) at (-1,-1.5) {};
\node (5) at (1,-1.5) {};
\coordinate [label=left : $X$] (X) at (-1.5,-1.5);

\draw[-] (2) to (5);

\end{tikzpicture}
\end{figure}
\noindent
Applying the foregoing construction yields the following picture corresponding to the right side of the previous commutative diagram:

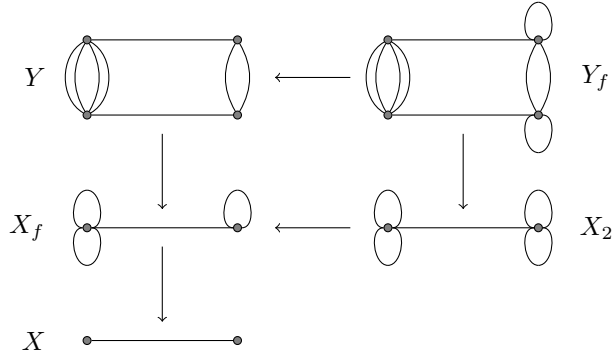
\begin{figure}[h]
\centering
\begin{tikzpicture}
\tikzstyle{every node}=[circle, draw, fill=black!50, inner sep=0pt, minimum width=3pt]

\node (1) at  (-1,1)  {};
\coordinate [label=left : $Y$] (Y) at (-1.5,.5);
\node (3) at  (1,1) {};

\node (4) at  (-1,0)  {};
\node (6) at  (1,0) {};

\draw [-] (3) to (1);

\draw[-] (6) to (4);

\draw[-] (1) to [out=240, in=120] (4);
\draw[-] (1) to [out=300, in=60] (4);
\draw[-] (1) to [out=210, in=150] (4);
\draw[-] (1) to [out=330, in=30] (4);
\draw[-] (3) to [out=240, in=120] (6);
\draw[-] (3) to [out=300, in=60] (6);

\draw[->] (0,-.25) -- (0,-1.25);

\node (2) at (-1,-1.5) {};
\node (5) at (1,-1.5) {};
\coordinate [label=left : $X_f$] (Xf) at (-1.5,-1.5);

\draw[-] (2) to (5);
\draw[-] (2) to [out=0, in=0] (-1,-1);
\draw[-] (2) to [out=180, in=180] (-1,-1);
\draw[-] (2) to [out=0, in=0] (-1,-2);
\draw[-] (2) to [out=180, in=180] (-1,-2);
\draw[-] (5) to [out=0, in=0] (1,-1);
\draw[-] (5) to [out=180, in=180] (1,-1);

\draw[->] (0,-1.75) -- (0,-2.75);

\node (7) at (-1,-3) {};
\node (8) at (1,-3) {};
\coordinate [label=left : $X$] (X) at (-1.5,-3);

\draw[-] (7) to (8);

\draw[<-] (1.5,.5) -- (2.5,.5);
\draw[<-] (1.5,-1.5) -- (2.5,-1.5);

\node (9) at  (3,1)  {};
\coordinate [label=right : $Y_f$] (Yf) at (5.5,.5);
\node (10) at  (5,1) {};

\node (11) at  (3,0)  {};
\node (12) at  (5,0) {};

\draw [-] (10) to (9);

\draw[-] (12) to (11);

\draw[-] (9) to [out=240, in=120] (11);
\draw[-] (9) to [out=300, in=60] (11);
\draw[-] (9) to [out=210, in=150] (11);
\draw[-] (9) to [out=330, in=30] (11);
\draw[-] (10) to [out=240, in=120] (12);
\draw[-] (10) to [out=300, in=60] (12);
\draw[-] (10) to [out=0, in=0] (5,1.5);
\draw[-] (10) to [out=180, in=180] (5,1.5);
\draw[-] (12) to [out=0, in=0] (5,-.5);
\draw[-] (12) to [out=180, in=180] (5,-.5);

\draw[->] (4,-.25) -- (4,-1.25);

\node (13) at (3,-1.5) {};
\node (14) at (5,-1.5) {};
\coordinate [label=right : $X_2$] (X2) at (5.5,-1.5);

\draw[-] (13) to (14);
\draw[-] (13) to [out=0, in=0] (3,-1);
\draw[-] (13) to [out=180, in=180] (3,-1);
\draw[-] (13) to [out=0, in=0] (3,-2);
\draw[-] (13) to [out=180, in=180] (3,-2);
\draw[-] (14) to [out=0, in=0] (5,-1);
\draw[-] (14) to [out=180, in=180] (5,-1);
\draw[-] (14) to [out=0, in=0] (5,-2);
\draw[-] (14) to [out=180, in=180] (5,-2);

\end{tikzpicture}
\end{figure}

\noindent
\underline{Step 2:}
Let $\textbf{\'Et}_{(X,x)}(n)$ be the full subcategory of $\textbf{\'Et}_{(X,x)}$ consisting of maps $f:(G,Y,y)\rightarrow (X,x)$ such that $X_f$ has at most $n$ loops at each vertex, hence is dominated by the graph $X_n$ as above. Our construction defines a functor $\Phi:\textbf{\'Et}_{(X,x)}(n)\rightarrow TL_n \backslash \textbf{GalCov}_{(X_n,x)}$, where the latter category consists of finite pointed Galois topological covers of $(X_n,x)$, considered up to the action of the topological loop group $TL_n$. We claim that $\Phi$ is an isomorphism of categories.

Indeed, let $(\mathcal{Y},y)\rightarrow (X_n,x)$ be a Galois topological cover, and set $G=\textrm{Aut}(\mathcal{Y}|X_n)$.  Let $Y:=\mathcal{Y} - \{\textrm{loops in $\mathcal{Y}$}\}$, and observe that $(G,Y)$ is an \'etale action inducing an isomorphism $G\backslash Y\rightarrow X$. Moreover, $\widetilde{G\backslash Y}$ has at most $n$ loops at each vertex. This construction provides an inverse functor to $\Phi$, and we have established the isomorphism of categories $\textbf{\'Et}_{(X,x)}(n)\cong TL_n \backslash \textbf{GalCov}_{(X_n,x)}.$\\

\noindent
\underline{Step 3:}
$\textbf{\'Et}_{(X,x)}$ is the direct limit of the subcategories $\textbf{\'Et}_{(X,x)}(n)$, so we obtain an isomorphism
$$
\textbf{\'Et}_{(X,x)}\rightarrow\varinjlim_{n} TL_n \backslash \textbf{GalCov}_{(X_n,x)}
=TL_\infty \backslash \textbf{GalCov}^*_{(X_\infty,x)},
$$
where $X_\infty=\varprojlim_{n}X_n$ is the graph $X$ with countably-many loops at each vertex, $TL_\infty=\prod_{V(X)}\varinjlim_{n}\left((\mathbb{Z}/2\mathbb{Z})^n\rtimes \mathfrak{S}_n\right)$, and $\textbf{GalCov}^*_{(X_\infty,x)}$ is the category of finite pointed Galois covers $(\mathcal{Y},y)\rightarrow (X_{\infty},x)$ with the property that all but finitely many loops of $X_{\infty}$ lift to loops in $\mathcal{Y}$.\\

\noindent
\underline{Step 4:}
The fiber functor at $x$ yields an equivalence of categories
$\textrm{Fib}_x:\textbf{Cov}_{(X_\infty,x)}\rightarrow \widehat{\pi_1(X_\infty,x)}-\textbf{FPSets}$, where $\textbf{Cov}_{(X_\infty,x)}$ is the category of finite pointed topological covers, $\widehat{\pi_1(X_{\infty},x)}$ denotes the profinite completion of the topological fundamental group $\pi_1(X_\infty,x)$, and \textbf{FPSets} stands for finite pointed sets. Observe that the spanning tree $T\subset X$ specifies a van Kampen isomorphism 
$$
\rho_T:\pi_1(X_\infty,x) \ \tilde{\rightarrow} \ \pi_1(X,x)\coprod\coprod_{V(X)}\pi_1(R_\infty),
$$
where $R_\infty$ is the rose with countably-many loops, and $\coprod$ denotes the coproduct in the category of groups, i.e. the free product. Moreover, the group $\pi_1(R_\infty)$ is free of countable rank, and has a canonical system of generators, well defined up to inverses.

Now $(\mathcal{Y},y)\rightarrow (X_{\infty},x)$ is in $\textbf{Cov}^*_{(X_\infty,x)}$ if and only if all but finitely many of the loops in $X_{\infty}$ lift to loops in $\mathcal{Y}$ if and only if all but finitely many of the canonical generators of the rose subgroups $\rho_T^{-1}(\pi_1(R_\infty))\le\pi_1(X_{\infty},x)$ are in the kernel of the associated monodromy representation. It follows that the fiber functor restricts to an equivalence of categories $\textrm{Fib}_x:\textbf{Cov}^*_{(X_\infty,x)}\rightarrow\hat{F}(\pi_1(X_\infty,x))-\textbf{FPSets}$, where $\hat{F}(\pi_1(X_\infty,x))$ denotes the \emph{free} profinite completion with respect to the roses, i.e. the inverse limit with respect to the system of normal subgroups of finite index containing all but finitely many of the canonical generators of the rose subgroups $\rho_T^{-1}(\pi_1(R_\infty))$. Define the \emph{\'etale fundamental group} of $(X,x)$ to be $\pi_1^\textrm{\'et}(X,x):=\hat{F}(\pi_1(X_\infty,x))$.

The preceding equivalence descends to an equivalence between the quotient categories under the actions of $TL_\infty$: 
$$
\textrm{Fib}_x:TL_\infty\backslash\textbf{Cov}^*_{(X_\infty,x)}\rightarrow (\pi_1^\textrm{\'et}(X,x)-\textbf{FPSets})/TL_{\infty}.
$$
Note that $\textrm{Fib}_x$ transforms the left action of $TL_\infty$ on covers into a right action on sets. Since finite Galois covers correspond to the coset spaces of open normal subgroups, we have an equivalence of categories
$$
TL_\infty\backslash\textbf{GalCov}^*_{(X_\infty,x)}\cong
TL_\infty\backslash\textbf{OpNor}_{\pi_1^\textrm{\'et}(X,x)},
$$
where the objects of the latter category are the orbits of the open normal subgroups of $\pi_1^\textrm{\'et}(X,x)$ under the left action of $TL_\infty$, and the morphisms are given by
$$
\textrm{Hom}(TL_\infty N,TL_\infty N')=\begin{cases}
\star& \text{if $\sigma N\subset N'$  for some $\sigma\in TL_\infty$},\\
\emptyset& \text{otherwise}.
\end{cases}
$$
But this category is in turn equivalent to $\textbf{FinSurj}(\pi_1^\textrm{\'et}(X,x))/TL_\infty$ of $TL_\infty$-orbits of surjections onto finite groups, where the Hom-sets are either empty or singletons as before. Putting all of this together, we obtain an equivalence of categories which we record in the following theorem.
\begin{theorem}\label{Et} There exists an equivalence of categories
$$
\textbf{\'Et}_{(X,x)}\rightarrow \textbf{FinSurj}(\pi_1^\textrm{\'et}(X,x))/TL_\infty,
$$
where the topological loop group $TL_\infty$ acts on the right by pre-composition with homomorphisms from $\pi_1^\textrm{\'et}(X,x)$. 
\end{theorem}

\section{Concrete description of \'etale $G$-covers}\label{concrete}

We now utilize the equivalence of Theorem~\ref{Et} to provide a concrete description of \'etale $G$-covers of a connected pointed graph $(X,x)$. Recall from the previous section that the spanning tree $T\subset X$ provides a van Kampen isomorphism $\rho_T:\pi_1^\textrm{\'et}(X,x)\cong\widehat{\pi_1(X,x)}\coprod\coprod_{V(X)}\hat{F}(\pi_1(R_\infty))$. Moreover, the group $\hat{F}(\pi_1(R_\infty))$ is isomorphic to the free profinite completion of the free group on countably many generators, i.e. the inverse limit with respect to the system of normal subgroups of finite index containing all but finitely many of the generators. 

By Theorem~\ref{Et}, to give a pointed \'etale $G$-cover $(G,Y,y)\rightarrow (X,x)$ is to give a surjective homomorphism from $\pi_1^\textrm{\'et}(X,x)$ onto $G$, considered up to pre-composition with elements of $TL_\infty$. But via the isomorphism $\rho_T$, such a homomorphism is the same as a homomorphism $\pi_1(X,x)\rightarrow G$ (yielding a topological $G$-cover of $X$) together with a finite multi-set of nontrivial elements of $G$ for each vertex of $X$, and the $TL_\infty$-action means that these multi-sets are symmetric and unordered. Clearly, the homomorphism from $\pi_1^\textrm{\'et}(X,x)$ will be surjective exactly when the image of the homomorphism from $\pi_1(X,x)$ together with the union of the multi-sets generate $G$. If $S_z=\{\delta_i\}$ is the multi-set attached to $z\in V(X)$, then the fiber $(G,Y_z,\tilde{z})$ of the corresponding $G$-cover is isomorphic to the Cayley graph $\textrm{Cay}(G,S_z)$ (see Example~\ref{Cay}). Here,  $\tilde{z}$ denotes the vertex of $Y_z$ determined by the unique lifting $\tilde{T}$ of the tree $T\subset X$ to $(Y,y)$ (see section~\ref{trees}). Moreover, the decomposition group at $\tilde{z}$ is the subgroup $\Delta_{\tilde{z}}$ generated by $S_z$ (see Definition~\ref{DI}). If $\Delta_{\tilde{z}}$ is a proper subgroup of $G$, then the disconnected fiber $(G,Y_z,\tilde{z})$ is obtained by induction from a connected Cayley graph on $\Delta_{\tilde{z}}$:
$$
(G,Y_z,\tilde{z})\cong\textrm{Ind}_{\Delta_{\tilde{z}}}^G\textrm{Cay}(\Delta_{\tilde{z}},S_z).
$$

Thus, we see that an \'etale $G$-cover of $X$ is a family of Cayley graphs on $G$ over the base $X$. The case of a topologically unramified cover corresponds to a family of trivial Cayley graphs on $G$, each with vertex set $G$ and no edges.  If we ignore basepoints in our covers, then an \'etale $G$-cover of $X$ corresponds to a surjection from $\pi_1^{\textrm{\'et}}(X,x)$ onto $G$, considered up to inner automorphisms of $G$. Hence, to give such a cover is to give the data described above, up to uniform conjugation by elements of $G$.

\section{Branched $G$-covers and inertia structures}

We now extend our classification to arbitrary harmonic $G$-covers of $X$. For this, fix a finite group $G$, and consider the full subcategory $\textbf{Harm}^G_{(X,x)}$ of $\textbf{Harm}_{(X,x)}$ consisting of $G$-covers of $X$ for the particular group $G$.

\begin{defn}
A \emph{$G$-inertia structure} on $X$ is a collection of subgroups $\mathcal{I}=\{I_z\le G \ | \ z\in V(X)\}$.
\end{defn}

 If $(G,Y,y)\rightarrow (X,x)$ is an \'etale $G$-cover, then as explained in the preceding section, each fiber $(G,Y_z,\tilde{z})$ is canonically isomorphic to a Cayley graph $\textrm{Cay}(G,S_z)$. In particular, we have a canonical identification $V(Y_z)=G$, which defines a \emph{right} action of the subgroup $I_z$ on $V(Y_z)$, given by right multiplication of $I_z$ on $G$. Let $Y^{I_z}$ be the graph obtained from $Y$ by identifying the right orbits of $I_z$  on $V(Y_z)$, and deleting any resulting loops. Then $G$ acts harmonically on $Y^{I_z}$ with quotient $X$ and inertia $I_z$ at the image $\tilde{z}'$ of $\tilde{z}$ in $Y^{I_z}$. Collapsing all of the fibers in this way, we obtain a harmonic $G$-cover $(G,Y^\mathcal{I},y')\rightarrow (X,x)$, together with a lifting $\tilde{T}^\mathcal{I}\subset Y^\mathcal{I}$ of the tree $T\subset X$, given by the image of the unique lifting $\tilde{T}\subset (Y,y).$ By construction, the inertia group at $\tilde{z}'\in V(Y^\mathcal{I})$ is the given subgroup $I_z\in\mathcal{I}$.

We thus have a functor 
$$
\mathcal{F}^\mathcal{I}:\textbf{\'Et}^G_{(X,x)}\rightarrow \textbf{Harm}^{G,C(\mathcal{I})}_{(X,x)}
$$
from the full subcategory of pointed \'etale $G$-covers of $(X,x)$ to the full subcategory of pointed harmonic $G$-covers $(G,Y,y)\rightarrow (X,x)$ with inertia groups given by the conjugacy classes 
$C(\mathcal{I}):=\{c(I_z) \ | \ z\in V(X)\}$ such that the inertia group at $y$ is $I_x$.

\begin{prop}\label{Harm}
 Every object of $\textbf{Harm}^{G,C(\mathcal{I})}_{(X,x)}$ is in the image of a functor $\mathcal{F}^{\tilde{\mathcal{I}}}$ where $\tilde{\mathcal{I}}=\{g_zI_zg_z^{-1} \ | \ z\in V(X)-\{x\}\}\cup\{I_x\}$ is a pointwise conjugate of $\mathcal{I}$ away from $x$.
 \end{prop}
 
 \begin{proof}
Let $f:(G,Y,y)\rightarrow (X,x)$ be a harmonic $G$-cover with inertia given by the conjugacy classes $C(\mathcal{I})$ such that $I_y=I_x$. Choose a lifting $\tilde{T}\subset (Y,y)$ of $T\subset X$, which defines a $G$-inertia structure $\tilde{\mathcal{I}}=\{I_{\tilde{z}} \ | \ z\in V(X)\}\in C(\mathcal{I})$ on $X$. For each $z\in V(X)$, we have an isomorphism of left $G$-sets $G/I_{\tilde{z}} \rightarrow V(Y_z)$ defined by sending $gI_{\tilde{z}}$ to $g\tilde{z}$. Moreover, since the $G$-action is harmonic, the inertia group $I_{\tilde{z}}$ acts freely on the edges adjacent to $\tilde{z}$ in $Y_z$, so those edges may be labeled by a multi-set $\overline{S}_z:=\{I_{\tilde{z}}\delta_i I_{\tilde{z}}\}$ of left $I_{\tilde{z}}$-orbits of left cosets of $I_{\tilde{z}}$. Here we must count appropriately: not only can each orbit appear more than once (accounting for multiple edges), but every orbit corresponds to $|I_{\tilde{z}}|$ edges. 

Observe that if $e\in E(Y_z)$ connects $I_{\tilde{z}}$ to $\delta I_{\tilde{z}}$, then the edge $\delta^{-1} e$ connects $I_{\tilde{z}}$ to $\delta^{-1} I_{\tilde{z}}$. This implies that the multi-set $\overline{S}_z$ is symmetric: the multiplicity of $I_{\tilde{z}}\delta_i I_{\tilde{z}}$ in $\overline{S}_z$ is the same as the multiplicity of $I_{\tilde{z}}\delta_i^{-1} I_{\tilde{z}}$. Of course, it is possible that $I_{\tilde{z}}\delta_i I_{\tilde{z}}=I_{\tilde{z}}\delta_i^{-1}I_{\tilde{z}}$ for some $i$, but the multiplicity of such a self-inverse orbit in $\overline{S}_z$ is automatically even since the $G$-action is unflipped. Indeed, for a self-inverse orbit, we may choose a representative $\delta$ so that $\delta I_{\tilde{z}}=\delta^{-1} I_{\tilde{z}}$. If $e$ connects $I_{\tilde{z}}$ to $\delta I_{\tilde{z}}$, then so does the distinct edge $\delta e$. But then the two orbits $I_{\tilde{z}}e$ and $I_{\tilde{z}}\delta e$ are distinct and together contribute 2 to the multiplicity of $I_{\tilde{z}}\delta I_{\tilde{z}}$ in $\overline{S}_z$. 

The symmetry of $\overline{S}_z$ allows us to define a symmetric multi-set of elements from $G$, given by $S_z:=\cup_i\{\delta_i,\delta_i^{-1}\}$, where for each inverse pair of orbits from $\overline{S}_z$, we have chosen an inverse pair of representatives. Note that if the orbit $I_{\tilde{z}}\delta_i I_{\tilde{z}}$ can be represented by an involution $\delta_i$, then it is self-inverse, and occurs with even multiplicity by the previous paragraph. In this special case, we take each \emph{pair} of orbits $I_{\tilde{z}}\delta_i I_{\tilde{z}}$ to contribute one copy of the involution $\delta_i$ to $S_z$.

Next, we choose a finite set of edges $\{\xi_j\}$ adjacent to $\tilde{z}$ in $Y-Y_z$ such that $f^{-1}(E(z(1)))$ is the disjoint union of the $G$-orbits of the $\xi_j$ (each of which has size $|G|$ since the action is harmonic). Furthermore, we may choose the $\xi_j$ so as to include the edges of the tree $\tilde{T}\subset Y$ adjacent to $\tilde{z}$. Define a new graph $Y'$ by glueing  $Y-Y_z$ to the Cayley graph $\textrm{Cay}(G,S_z\cup (I_{\tilde{z}}-\{\varepsilon\}))$ (see Example~\ref{Cay}) by identifying the vertex $g\in G$ with the free endpoint  of the edge $g\xi_j$ for all $j$. The resulting graph $Y'$ supports a natural unflipped harmonic $G$-action, and the corresponding harmonic $G$-cover $(G,Y')\rightarrow X$ is \'etale at $z$. Repeating this procedure for all vertices of $X$, we finally obtain an \'etale $G$-cover of $X$, which we denote by $f^\textrm{\'et}:(G,Y^\textrm{\'et}, y^\textrm{\'et})\rightarrow (X,x)$. The inclusion of the non-identity inertia elements in the Cayley graphs guarantees that $Y^\textrm{\'et}$ is connected. Here $y^\textrm{\'et}$ is the vertex of $Y^\textrm{\'et}_x$ corresponding to $\textrm{id}_G$ in the Cayley graph construction starting with the inertia group $I_y=I_x\in\tilde{\mathcal{I}}$. By construction, we have $\mathcal{F}^{\tilde{\mathcal{I}}}(f^\textrm{\'et})=f$.
\end{proof}

The functors $\mathcal{F}^\mathcal{I}$ are neither full nor faithful. For instance, two \'etale covers that differ by vertical edges corresponding to inertia elements in $\mathcal{I}$ will be sent to the same harmonic cover by $\mathcal{F}^\mathcal{I}$.
Since our main concern in the next section is with existence theorems for harmonic $G$-covers, this will not concern us. We end this section by illustrating the construction of Proposition~\ref{Harm} via a continuation of Example~\ref{S3}.

\begin{ex}[An $\mathfrak{S}_3$-cover with inertia of order 2]\label{S3inertia}
Consider the symmetric group $\mathfrak{S}_3=\left<\sigma,\tau \ | \ \sigma^3=\tau^2=\varepsilon, \sigma\tau=\tau\sigma^2\right>$, acting as the permutation group of the double-sided triangular graph pictured below. (We have doubled the edges to permit an unflipped action.) As in the proof of Proposition~\ref{Harm}, the vertices have been labeled by the left-cosets of the inertia group $\left<\tau\right>$.

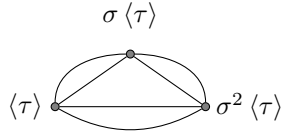
\begin{figure}[h]
\centering
\begin{tikzpicture}
\tikzstyle{every node}=[circle, draw, fill=black!50, inner sep=0pt, minimum width=3pt]

\node (4) at  (-1,0)  {};
\coordinate [label=left : $\left<\tau\right>$] (epsilon) at (-1.1,0);
\node (5) at  (0,.7) {} ;  
\coordinate [label=above : $\sigma\left<\tau\right>$] (sigma) at (0,.8);      
\node (6) at  (1,0) {};
\coordinate [label=right : $\sigma^2\left<\tau\right>$] (sigma2) at (1.1,0);

\draw[-] (4) to [out=90, in=180] (5);
\draw[-] (4) to (5);
\draw[-] (5) to [out=0, in=90] (6);
\draw[-] (5) to (6);
\draw[-] (6) to [out=210, in=330] (4);
\draw[-] (6) to (4);

\end{tikzpicture}
\end{figure}

\noindent
Following the proof of Proposition~\ref{Harm}, the multi-set $\overline{S}$ contains two orbits: $\overline{S}=\{\left<\tau\right>\sigma\left<\tau\right>, \left<\tau\right>\sigma^2\left<\tau\right>\}$, from which we obtain the multi-set $S=\{\sigma,\sigma^2\}=\{\sigma,\sigma^{-1}\}$. In order to obtain a connected \'etale group action, we construct the Cayley graph $\textrm{Cay}(\mathfrak{S}_3,S\cup\{\tau\})$ shown below, which gives an \'etale $\mathfrak{S_3}$-cover of the point-graph $\star$ as described in Example~\ref{S3}.

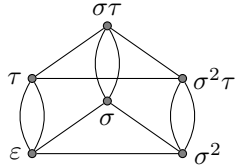
\begin{figure}[h]
\centering
\begin{tikzpicture}
\tikzstyle{every node}=[circle, draw, fill=black!50, inner sep=0pt, minimum width=3pt]

\node (1) at  (-1,1)  {};
\coordinate [label=left : $\tau$] (tau) at (-1.1,1);
\node (2) at  (0,1.7) {} ;  
\coordinate [label=above : $\sigma\tau$] (sigmatau) at (0,1.7);      
\node (3) at  (1,1) {};
\coordinate [label=right : $\sigma^2\tau$] (sigma2tau) at (1.1,1);

\node (4) at  (-1,0)  {};
\coordinate [label=left : $\varepsilon$] (epsilon) at (-1.1,0);
\node (5) at  (0,.7) {} ;  
\coordinate [label=below : $\sigma$] (sigma) at (0,.6);      
\node (6) at  (1,0) {};
\coordinate [label=right : $\sigma^2$] (sigma2) at (1.1,0);

\draw[-] (1) to (2);
\draw[-] (2) to (3);
\draw [-] (3) to (1);

\draw[-] (4) to (5);
\draw[-] (5) to (6);
\draw[-] (6) to (4);

\draw[-] (1) to [out=240, in=120] (4);
\draw[-] (1) to [out=300, in=60] (4);
\draw[-] (2) to [out=240, in=120] (5);
\draw[-] (2) to [out=300, in=60] (5);
\draw[-] (3) to [out=240, in=120] (6);
\draw[-] (3) to [out=300, in=60] (6);

\end{tikzpicture}
\end{figure}
\noindent
Applying the functor $\mathcal{F}^{\left<\tau\right>}$ to this cover yields the original $\mathfrak{S}_3$-action on the double-sided triangle.
\end{ex}

\section{Embedding problems and a Grunwald-Wang theorem}\label{EP}

In this section we consider embedding problems for finite graphs (see \cite{Har} section 5.1 for more on embedding problems in the context of function fields of curves). The following theorem states that finite embedding problems for graphs always have proper solutions, with good control over the branch locus.

\begin{theorem}\label{Embed}
Suppose that $f:(G,Y,y)\rightarrow (X,x)$ is a connected\footnote{In this section we consistently emphasize the connectedness of the covers, even though this is redundant given our definition of harmonic $G$-covers in section~\ref{HG}.} harmonic $G$-cover, and consider an exact sequence of finite groups
$$
1\rightarrow K\rightarrow G'\xrightarrow{\rho} G\rightarrow 1.
$$
Then there exists a connected harmonic $G'$-cover of $(X,x)$ dominating $f$ via the group homomorphism $\rho$. If $f$ is \'etale, then there exists a dominating \'etale $G'$-cover. Moreover, if the branch locus (horizontal and vertical) of $f$ is nonempty, then there exists a dominating $G'$-cover with the same branch locus as $f$.
\end{theorem}

\begin{proof}
By Proposition~\ref{Harm} (and its proof), there exists a connected \'etale $G$-cover $f^\textrm{\'et}:(G,Y^\textrm{\'et},y^\textrm{\'et})\rightarrow (X,x)$ and a  $G$-inertia structure $\mathcal{I}$ on $(X,x)$ such that $f=\mathcal{F}^\mathcal{I}(f^\textrm{\'et})$. Furthermore, by Theorem~\ref{Et}, the cover $f^\textrm{\'et}$ corresponds to a surjection $\phi:\pi_1^\textrm{\'et}(X,x)\rightarrow G$ (up to the action of $TL_\infty$), and we wish to lift $\phi$ through $\rho$ to a surjection $\phi':\pi_1^\textrm{\'et}(X,x)\rightarrow G'$. Such a lifting always exists since $\pi_1^\textrm{\'et}(X,x)$ is free profinite of countably-infinite rank (see section~\ref{concrete}). Each such lifting (up to the action of $TL_\infty$) yields a connected \'etale $G'$-cover $f^{'\textrm{\'et}}:(G',Y^{'\textrm{\'et}},y^{'\textrm{\'et}})\rightarrow (X,x)$ dominating $f^\textrm{\'et}$ via $\rho$. Moreover, assuming that $Y\rightarrow X$ is \emph{not} a topologically unramified cover, we can choose the lifting $\phi'$ so that points of $X$ that are totally split in $Y^\textrm{\'et}$ remain totally split in $Y^{'\textrm{\'et}}$.

Now let $\mathcal{I'}$ be any $G'$-inertia structure on $(X,x)$ with the property that $\rho(I'_z)=I_z\in\mathcal{I}$ for all $z\in V(X)$. Such an inertia structure always exists: take $\mathcal{I'}:=\rho^{-1}(\mathcal{I})$ for instance. Set 
$(G',Y',y'):=\mathcal{F}^\mathcal{I'}(G',Y^{'\textrm{\'et}},y^{'\textrm{\'et}})$, and observe that this connected harmonic $G'$-cover dominates the $G$-cover $(G,Y,y)$. In the case where $\mathcal{I'}:=\rho^{-1}(\mathcal{I})$, note that every point of $Y$ is totally ramified in $Y'$. At the other extreme, if there exists a section $\sigma:G\rightarrow G'$ to $\rho$, then taking $\mathcal{I'}:=\sigma(\mathcal{I})$ yields an \'etale $K$-cover $Y'\rightarrow Y$. Of course, if the point $z\in V(X)$ is horizontally unramified in $Y\rightarrow X$, then $I_z=\{\textrm{id}_G\}$, and we can always choose $I'_z=\{\textrm{id}_{G'}\}$, so that $z$ is still horizontally unramified in $Y'\rightarrow X$.
\end{proof}

Given a finite group $G$ and a connected graph $X$, it is easy to construct a connected \'etale $G$-cover of $X$: start with $|G|$ disjoint copies of $X$ (labeled by the elements of $G$), and connect them with vertical fibers given by connected Cayley graphs $\textrm{Cay}(G,S)$ (see example~\ref{Cay}). But our analysis shows much more than an affirmative answer to this existence problem for \'etale $G$-covers of graphs. As motivation, recall the Grunwald-Wang theorem concerning abelian extensions of a global field $k$ (see \cite{NSW}): let $S$ be a finite set of primes of $k$, and let $A$ be an abelian group. For each $\mathfrak{p}\in S$, let $K_\mathfrak{p}|k_\mathfrak{p}$ be an abelian extension of the completion $k_\mathfrak{p}$ with Galois group isomorphic to a subgroup of $A$. Then (except in one special case), there exists an $A$-Galois extension $K|k$ inducing the given local extensions $K_\mathfrak{p}|k_\mathfrak{p}$ for all $\mathfrak{p}\in S$. For graphs we have:
\begin{theorem}\label{GW}
Let $X$ be a finite connected graph of genus $g$, $B\subset V(X)$ a subset of vertices, and $G$ a finite group.  For each $b\in B$, let $(G_b,Y_b)\rightarrow b$ be a connected harmonic $G_b$-cover of the point $b$, where $G_b$ is isomorphic to a subgroup of $G$.  For each $b\in B$, choose an embedding $\varphi_b:G_b\hookrightarrow G$, and let $G(B)$ be the subgroup of $G$ generated by the images $\varphi_b(G_b)$ for $b\in B$. Assume that  $G=\left<G(B),\gamma_1,\dots,\gamma_g\right>$ for some elements $\gamma_i\in G$. Then there exists a connected harmonic $G$-cover $(G,\mathcal{Y})\rightarrow X$, totally split outside of $B$, such that for each $b\in B$, the fiber $(G,\mathcal{Y}_b)\rightarrow b$ is isomorphic to the $G$-cover of $b$ induced by the given $G_b$-cover $(G_b,Y_b)\rightarrow b$.
\end{theorem}
\begin{proof}
For each $b\in B$, we are given a transitive, harmonic $G_b$-action $(G_b,Y_b,y_b)$, where we have chosen a point $y_b\in V(Y_b)$. By Proposition~\ref{Harm} and Theorem~\ref{Et}, these actions correspond to inertia groups $I_b\subset G_b$ and multi-sets $\{\delta_i^{(b)}\}\subset G_b$. Via the embeddings $\varphi_b$, we may view everything in the group $G$. Since $X$ has genus $g$, there exists a homomorphism $\pi_1(X,x)\rightarrow G$ defined by sending a free basis to the elements $\gamma_i$. This homomorphism together with the multi-sets $\{\delta_i^{(b)}\}$ give a homomorphism $\pi_1^\textrm{\'et}(X,x)\rightarrow G$, and the inertia groups $I_b$ give a $G$-inertia structure $\mathcal{I}$ on $(X,x)$, where we set $I_z=\{\textrm{id}_G\}$ if $z\not\in B$. Together these define a harmonic $G$-cover $(G,\mathcal{Y})\rightarrow X$, which is connected by our assumption on the generation of $G$. By construction, it is totally split outside of $B$, and the fibers over $B$ are as required.
\end{proof}


\end{document}